\documentclass[a4paper,11pt]{amsart}

\usepackage[margin=1in]{geometry}
\usepackage{amsmath,amssymb,amsthm,mathtools}
\usepackage{microtype}
\usepackage[hidelinks]{hyperref}

\numberwithin{equation}{section}

\newtheorem{theorem}{Theorem}[section]
\newtheorem{proposition}[theorem]{Proposition}
\newtheorem{lemma}[theorem]{Lemma}
\newtheorem{corollary}[theorem]{Corollary}
\theoremstyle{definition}
\newtheorem{definition}[theorem]{Definition}
\theoremstyle{remark}
\newtheorem{remark}[theorem]{Remark}

\newcommand{\1}{\mathbf 1}
\newcommand{\cF}{\mathcal F}
\newcommand{\cG}{\mathcal G}

\newcommand{\wtLambda}{\widetilde\Lambda}

\newcommand{\lcmop}{\operatorname{lcm}}

\title{The Exponent of Harmonic LCM Avoidance}
\author{Yanping Luo$^{1}$}
\address{$^{1}$Sichuan University, Chengdu, China}
\email{2024222010045@stu.scu.edu.cn}

\author{Ruiyi Yang$^{2}$}
\address{$^{2}$Fudan University, Shanghai, China}
\email{2330709018@m.fudan.edu.cn}

\author{Keheng Zhu$^{3}$}
\address{$^{3}$Capital Normal University, Beijing, China}
\email{2240502168@cnu.edu.cn}


\begin{document}

\begin{abstract}
Fix $k\ge 3$, and let $f_k(N)$ be the largest harmonic sum of a subset of $[N]$ containing no $k$ distinct integers with a common pairwise least common multiple.  We prove that $f_k(N)=(\log N)^{\gamma_k+o(1)}$ for a well-defined exponent $\gamma_k\in(0,1]$.  Following the weighted-pressure idea of Chojecki, we give a self-contained proof of variational formulas for $\gamma_k$ in terms of weighted sunflower-free families.  We then eliminate the continuous weight: if $M_k(n,r)$ is the largest size of an $r$-uniform $k$-cosunflower-free family on $[n]$, then
\[
\gamma_k=\sup_{n\ge1,\ 1\le r\le n}\frac{r}{en}M_k(n,r)^{1/r}.
\]
This finite-block formula yields a direct transfer principle from uniform set-system constructions, recovers the Tang--Zhang bounds, and gives $0.438899\ldots<\gamma_3\le0.889881\ldots$.
\end{abstract}

\maketitle

\section{Introduction and main results}

Write $[N]=\{1,\dots,N\}$.  A set $A\subseteq[N]$ is called \emph{LCM-$k$-free} if it contains no distinct $a_1,\dots,a_k$ for which $\lcmop(a_i,a_j)$ is independent of the pair $1\le i<j\le k$.  Following Erd\H{o}s~\cite{Erdos1970}, define
\[
f_k(N):=\max\left\{\sum_{a\in A}\frac1a:A\subseteq[N]\text{ is LCM-$k$-free}\right\}.
\]
Erd\H{o}s proved $f_k(N)\ll_k\log N/\log\log N$.  Tang and Zhang~\cite{TangZhang2025} subsequently connected this problem to the Erd\H{o}s--Szemer\'edi weak sunflower problem and established nontrivial polylogarithmic upper and lower bounds.  In particular, if $\mu_k^{\mathrm S}$ denotes the exponential capacity of $k$-sunflower-free subfamilies of $2^{[n]}$, then
\begin{equation}\label{eq:TZ-bounds}
 (\log N)^{\log\mu_k^{\mathrm S}-o(1)}\le f_k(N)
 \ll (\log N)^{\mu_k^{\mathrm S}-1+o(1)}.
\end{equation}
A note of Chojecki~\cite{Chojecki2026}, posted in the discussion of Erd\H{o}s Problem~\#856~\cite{Bloom856}, introduced a weighted sunflower pressure and observed that the logarithmic exponent itself exists.  We give a detailed proof of that assertion and derive a finite uniform-block characterization which is not present in the cited formulations.

Recall that distinct sets $S_1,\dots,S_k$ form a \emph{$k$-sunflower} if their pairwise intersections are all equal, and a \emph{$k$-cosunflower} if their pairwise unions are all equal.  Let
\[
M_k(n,r):=\max\left\{|\cG|:\cG\subseteq\binom{[n]}r\text{ is $k$-cosunflower-free}\right\}.
\]
The following theorem is the main conclusion.

\begin{theorem}[Exact exponent and finite-block formula]\label{thm:main}
For every fixed integer $k\ge3$, the limit
\[
\gamma_k:=\lim_{N\to\infty}\frac{\log f_k(N)}{\log\log N}
\]
exists.  Equivalently,
\[
f_k(N)=(\log N)^{\gamma_k+o(1)}.
\]
Moreover,
\begin{equation}\label{eq:finite-block}
\displaystyle
\gamma_k=\sup_{\substack{n\ge1\\1\le r\le n}}
\frac{r}{en}M_k(n,r)^{1/r}.
\end{equation}
In particular, the optimal logarithmic exponent is the supremum of explicitly finite extremal quantities.
\end{theorem}

The finite-block formula gives both structural and numerical consequences.  Let $F_k(n)$ be the largest cardinality of a $k$-sunflower-free family in $2^{[n]}$ and put
\[
\mu_k^{\mathrm S}:=\lim_{n\to\infty}F_k(n)^{1/n};
\]
the existence of this limit will follow from the same tensor argument used below.

\begin{corollary}[Comparison with ordinary sunflower capacity]\label{cor:capacity}
For every $k\ge3$,
\begin{equation}\label{eq:capacity-comparison}
\max\left\{\frac{k-2}{e((k-2)!)^{1/(k-2)}},\ \log\mu_k^{\mathrm S}\right\}
\le \gamma_k\le \mu_k^{\mathrm S}-1\le1.
\end{equation}
Consequently,
\[
0.438899884\ldots<\gamma_3\le0.889881574\ldots.
\]
Furthermore, $\gamma_k$ is nondecreasing in $k$ and $\gamma_k\to1$ as $k\to\infty$; more precisely, with $r=k-2$,
\begin{equation}\label{eq:large-k}
\gamma_k\ge \frac{r}{e(r!)^{1/r}}
=1-\frac{\log(2\pi r)}{2r}+O\!\left(\frac{(\log r)^2}{r^2}\right).
\end{equation}
\end{corollary}

The value of $\gamma_k$ is not presently known numerically, even for $k=3$.  Thus Theorem~\ref{thm:main} settles the existence and exact characterization of the polylogarithmic exponent, while its evaluation remains a finite-uniform sunflower problem.

\section{Weighted sunflower pressures}

For $z>0$ and $n\ge0$ (with the convention $[0]=\varnothing$), define the weighted extremal partition functions
\begin{align*}
W_k(n;z)&:=\max\left\{\sum_{S\in\cF}z^{|S|}:\cF\subseteq2^{[n]}\text{ is $k$-sunflower-free}\right\},\\
C_k(n;z)&:=\max\left\{\sum_{S\in\cG}z^{|S|}:\cG\subseteq2^{[n]}\text{ is $k$-cosunflower-free}\right\}.
\end{align*}
We first record the tensor mechanism which makes exponential capacities well defined.

\begin{lemma}[Uniform tensor products]\label{lem:tensor}
Let $\mathcal U\subseteq\binom{[n]}r$ be $k$-sunflower-free, and place copies $\mathcal U^{(1)},\dots,\mathcal U^{(t)}$ on pairwise disjoint $n$-element blocks.  Then
\[
\mathcal U^{\boxtimes t}:=\{U_1\sqcup\cdots\sqcup U_t:U_j\in\mathcal U^{(j)}\}
\]
is $k$-sunflower-free.  The analogous assertion holds with ``sunflower'' replaced by ``cosunflower.''
\end{lemma}

\begin{proof}
Suppose that $B_1,\dots,B_k\in\mathcal U^{\boxtimes t}$ form a sunflower.  Restricting to any block gives $r$-element sets $B_1^{(j)},\dots,B_k^{(j)}$ with equal pairwise intersections.  If two of these restrictions coincide, say $B_1^{(j)}=B_2^{(j)}=B$, then the common pairwise intersection equals $B$; hence $B\subseteq B_i^{(j)}$ for every $i$, and uniformity forces all restrictions to equal $B$.  Otherwise the $k$ restrictions are distinct and form a forbidden sunflower in $\mathcal U^{(j)}$.  Thus all restrictions are equal in every block, so $B_1=\cdots=B_k$, contrary to distinctness.

For cosunflowers the argument is dual.  If two restrictions coincide with value $B$, then their union is $B$, so every other restriction is contained in $B$ and hence equals $B$ by uniformity.  The same blockwise conclusion follows.
\end{proof}

\begin{proposition}[Existence and duality of pressure]\label{prop:pressure}
For every $k\ge3$ and $z>0$, the limits
\[
\Lambda_k(z):=\lim_{n\to\infty}W_k(n;z)^{1/n},\qquad
\wtLambda_k(z):=\lim_{n\to\infty}C_k(n;z)^{1/n}
\]
exist.  They satisfy
\begin{equation}\label{eq:duality}
\wtLambda_k(z)=z\Lambda_k(1/z),\qquad \Lambda_k(1)=\wtLambda_k(1)=\mu_k^{\mathrm S},
\end{equation}
and
\begin{equation}\label{eq:trivial-pressure}
z\le\Lambda_k(z)\le1+z,
\qquad 1\le\wtLambda_k(z)\le1+z.
\end{equation}
\end{proposition}

\begin{proof}
Choose a $k$-sunflower-free family $\cF\subseteq2^{[n]}$ attaining $W_k(n;z)$.  One of its $n+1$ uniform layers has weighted mass at least $W_k(n;z)/(n+1)$.  Lemma~\ref{lem:tensor} therefore gives
\[
W_k(tn;z)\ge\left(\frac{W_k(n;z)}{n+1}\right)^t.
\]
Also $W_k(m;z)$ is nondecreasing in $m$, by embedding a family on $[m]$ into one on $[m+1]$ and never using the new point.  If $tn\le m<(t+1)n$, then
\[
W_k(m;z)^{1/m}\ge
\left(\frac{W_k(n;z)}{n+1}\right)^{t/m}.
\]
Taking $m\to\infty$, then taking the supremum over arbitrarily large $n$, shows that the liminf of $W_k(m;z)^{1/m}$ is at least its limsup.  Hence the limit exists.

Complementation in $[n]$ bijects cosunflower-free and sunflower-free families and transforms $z^{|S|}$ into $z^n(1/z)^{|[n]\setminus S|}$.  Thus
\[
C_k(n;z)=z^nW_k(n;1/z),
\]
which proves the existence of the second limit and~\eqref{eq:duality}.  Finally, the one-set family $\{[n]\}$ has weight $z^n$, while the whole power set has weight $(1+z)^n$; this gives the first pair of inequalities in~\eqref{eq:trivial-pressure}.  The second pair follows by duality.
\end{proof}

We shall use only elementary consequences of Mertens' estimates; see, for example, \cite{MontgomeryVaughan2006}
\begin{equation}\label{eq:mertens}
\sum_{p\le x}\frac1p=\log\log x+O(1),\qquad
\sum_{p\le x}\frac{\log p}{p}=\log x+O(1).
\end{equation}
Here and below $\omega(m)$ denotes the number of distinct prime factors of $m$, and $\mu$ is the M\"obius function.

\begin{lemma}[Two harmonic Euler-product estimates]\label{lem:analytic}
For every fixed $u,z>0$ and $X\ge3$,
\begin{equation}\label{eq:omega-upper}
\sum_{m\le X}\frac{u^{\omega(m)}}m\ll_u(\log X)^u
\end{equation}
and
\begin{equation}\label{eq:sqfree-asymp}
H_z(X):=\sum_{\substack{q\le X\\\mu(q)^2=1}}\frac{z^{\omega(q)}}q
\asymp_z(\log X)^z.
\end{equation}
\end{lemma}

\begin{proof}
Extending the first sum to all integers composed of primes at most $X$ gives
\[
\sum_{m\le X}\frac{u^{\omega(m)}}m
\le\prod_{p\le X}\left(1+\frac{u}{p-1}\right).
\]
For all sufficiently large $p$,
$\log(1+u/(p-1))=u/p+O_u(p^{-2})$, so~\eqref{eq:mertens} yields~\eqref{eq:omega-upper}.

For~\eqref{eq:sqfree-asymp}, the upper bound follows from
\[
H_z(X)\le\prod_{p\le X}\left(1+\frac zp\right)\ll_z(\log X)^z.
\]
For the lower bound, put $\theta=1/(8(1+z))$ and $Y=X^\theta$.  Normalize the weights $z^{\omega(q)}/q$ over squarefree $q$ whose prime factors are at most $Y$ to obtain a random variable $Q$.  Its normalizing factor is
\[
Z_Y=\prod_{p\le Y}\left(1+\frac zp\right)\asymp_z(\log Y)^z\asymp_z(\log X)^z,
\]
and independence of the prime choices gives
\[
\mathbb E\log Q
=\sum_{p\le Y}\frac{z\log p}{p+z}
\le z\sum_{p\le Y}\frac{\log p}{p}
=z\log Y+O_z(1)\le\frac18\log X+O_z(1).
\]
For large $X$, Markov's inequality therefore gives $\mathbb P(Q>X)\le1/4$.  At least three quarters of the mass $Z_Y$ is contributed by $q\le X$, proving the required lower bound.
\end{proof}

\section{Arithmetic transference}

We now prove matching weighted upper and lower principles.  Their combination is what forces a single logarithmic exponent.

\begin{theorem}[Weighted upper bound]\label{thm:upper}
For fixed $k\ge3$ and $z>0$,
\begin{equation}\label{eq:weighted-upper}
f_k(N)\le(\log N)^{\Lambda_k(z)-z+o(1)}.
\end{equation}
Equivalently,
\[
\limsup_{N\to\infty}\frac{\log f_k(N)}{\log\log N}
\le\Lambda_k(z)-z.
\]
\end{theorem}

\begin{proof}
Let $A\subseteq[N]$ be LCM-$k$-free.  For each integer $m\ge1$, let $P(m)$ be its set of prime divisors and define
\[
\cF_m:=\left\{S\subseteq P(m):m=a\prod_{p\in S}p\text{ for some }a\in A\right\}.
\]
We claim that $\cF_m$ is $k$-sunflower-free.  Otherwise let $S_1,\dots,S_k$ be a sunflower with kernel $K$, and write $a_i=m/\prod_{p\in S_i}p\in A$.  Distinctness of the $S_i$ implies distinctness of the $a_i$.  For every prime $p\mid m$,
\[
v_p(a_i)=v_p(m)-\1_{p\in S_i},
\]
and hence
\[
v_p(\lcmop(a_i,a_j))
=v_p(m)-\1_{p\in S_i\cap S_j}
=v_p(m)-\1_{p\in K}.
\]
Thus all pairwise least common multiples are equal, contradicting the hypothesis on $A$.

Multiply the desired harmonic sum by the squarefree kernel $H_z(N)$ from Lemma~\ref{lem:analytic}:
\[
T_z(N):=\left(\sum_{a\in A}\frac1a\right)H_z(N)
=\sum_{a\in A}\ \sum_{\substack{q\le N\\\mu(q)^2=1}}
\frac{z^{\omega(q)}}{aq}.
\]
For a pair $(a,q)$ put $m=aq$ and $S=P(q)$.  Then $m\le N^2$, $S\in\cF_m$, and the summand is $z^{|S|}/m$.  Therefore
\[
T_z(N)\le\sum_{m\le N^2}\frac1m\sum_{S\in\cF_m}z^{|S|}
\le\sum_{m\le N^2}\frac{W_k(\omega(m);z)}m.
\]
For every $\delta>0$, the definition of $\Lambda_k(z)$ supplies a constant $C=C(k,z,\delta)$ such that
$W_k(r;z)\le C(\Lambda_k(z)+\delta)^r$ for all $r\ge0$.  Lemma~\ref{lem:analytic} now gives
\[
T_z(N)\ll_{k,z,\delta}(\log N)^{\Lambda_k(z)+\delta}.
\]
Since $H_z(N)\gg_z(\log N)^z$, we obtain
\[
\sum_{a\in A}\frac1a\ll_{k,z,\delta}
(\log N)^{\Lambda_k(z)-z+\delta}.
\]
Taking the maximum over $A$ and then letting $\delta\downarrow0$ proves the theorem.
\end{proof}

For the lower bound we need a simple blow-up fact.

\begin{definition}
Let $U_1,\dots,U_t$ be pairwise disjoint nonempty sets.  For $G\subseteq[t]$, define
\[
\mathcal B(G;U_1,\dots,U_t)
:=\left\{\{x_i:i\in G\}:x_i\in U_i\text{ for every }i\in G\right\}.
\]
For $\cG\subseteq2^{[t]}$, put
\[
\mathcal B(\cG;U_1,\dots,U_t)
:=\bigcup_{G\in\cG}\mathcal B(G;U_1,\dots,U_t).
\]
\end{definition}

\begin{lemma}[Cosunflower-free blow-ups]\label{lem:blowup}
If $\cG\subseteq2^{[t]}$ is $k$-cosunflower-free and $k\ge3$, then
$\mathcal B(\cG;U_1,\dots,U_t)$ is $k$-cosunflower-free.
\end{lemma}

\begin{proof}
Suppose that distinct $B_1,\dots,B_k$ in the blow-up have equal pairwise unions, and let
$G_j=\{i:B_j\cap U_i\ne\varnothing\}$.  Then the pairwise unions $G_i\cup G_j$ are equal.  We claim that the $G_j$ are distinct.  If, say, $G_1=G_2$ while $B_1\ne B_2$, then in some common block $U_i$ the two sets choose different elements $x_1\ne x_2$.  Since $k\ge3$, compare with $B_3$.  Equality of $B_1\cup B_2$ with $B_1\cup B_3$ forces $B_3$ to choose $x_2$ in $U_i$, while equality with $B_2\cup B_3$ forces it to choose $x_1$.  This is impossible because a blow-up member chooses at most one element from each block.  Thus $G_1,\dots,G_k$ are distinct and form a $k$-cosunflower in $\cG$, a contradiction.
\end{proof}

\begin{lemma}[Greedy bucketing]\label{lem:bucket}
Let $w_1,\dots,w_s$ be positive numbers, each smaller than $\delta$, and suppose
$\sum_jw_j\ge t(z+\delta)$.  Then the indices can be partitioned partially into $t$ pairwise disjoint groups $I_1,\dots,I_t$ such that
\[
z\le\sum_{j\in I_i}w_j<z+\delta\qquad(1\le i\le t).
\]
\end{lemma}

\begin{proof}
Build the groups successively, adding unused weights until their sum first reaches $z$.  The overshoot is less than $\delta$.  Before the $i$th group is formed, the previous groups have total weight less than $(i-1)(z+\delta)$, so at least $(t-i+1)(z+\delta)\ge z$ remains.  The procedure therefore produces all $t$ groups.
\end{proof}

\begin{theorem}[Weighted lower bound]\label{thm:lower}
For fixed $k\ge3$ and $z>0$,
\begin{equation}\label{eq:weighted-lower}
\liminf_{N\to\infty}\frac{\log f_k(N)}{\log\log N}
\ge\frac{\log\wtLambda_k(z)}z.
\end{equation}
Equivalently,
\[
f_k(N)\ge(\log N)^{\log\wtLambda_k(z)/z-o(1)}.
\]
\end{theorem}

\begin{proof}
The assertion is trivial if $\wtLambda_k(z)=1$, so fix $1<\lambda<\wtLambda_k(z)$ and $0<\eta<1/2$.  Put
\[
L=\log\log N,\qquad
t=\left\lfloor\frac{(1-\eta)L}{z}\right\rfloor,
\qquad x=N^{1/t},\qquad y=\exp(L^{2/3}).
\]
As $N\to\infty$, we have $t\to\infty$ and $y<x$.  By the definition of $\wtLambda_k(z)$, for all sufficiently large $t$ there is a $k$-cosunflower-free family $\cG_t\subseteq2^{[t]}$ satisfying
\begin{equation}\label{eq:weighted-family}
\sum_{G\in\cG_t}z^{|G|}\ge\lambda^t.
\end{equation}
Choose a fixed $\delta$ with $0<\delta<\eta z/4$.  Mertens' theorem gives
\begin{align*}
\sum_{y<p\le x}\frac1p
&=\log\log x-\log\log y+O(1)\\
&=L-\log t-\frac23\log L+O(1)=L-o(L).
\end{align*}
On the other hand,
\[
t(z+\delta)
\le(1-\eta)L\left(1+\frac\delta z\right)+O(1)
\le(1-\eta/2)L+O(1).
\]
Also $1/p\le1/y<\delta$ for all primes $p>y$ once $N$ is large.  Lemma~\ref{lem:bucket} therefore supplies disjoint prime sets
$P_1,\dots,P_t\subseteq(y,x]$ such that
\begin{equation}\label{eq:bucket-mass}
J_i:=\sum_{p\in P_i}\frac1p\in[z,z+\delta)
\qquad(1\le i\le t).
\end{equation}

For $G\in\cG_t$, let
\[
A(G):=\left\{\prod_{i\in G}p_i:p_i\in P_i\text{ for each }i\in G\right\},
\qquad A:=\bigcup_{G\in\cG_t}A(G).
\]
Every $a\in A$ is squarefree and satisfies $a\le x^{|G|}\le x^t=N$.  The prime-support family of $A$ is the blow-up $\mathcal B(\cG_t;P_1,\dots,P_t)$, which is $k$-cosunflower-free by Lemma~\ref{lem:blowup}.  Since prime supports convert least common multiples of squarefree integers into unions, $A$ is LCM-$k$-free.

The sets $A(G)$ are pairwise disjoint, because the buckets are disjoint and the set of buckets meeting the prime support recovers $G$.  Hence~\eqref{eq:bucket-mass} and~\eqref{eq:weighted-family} give
\[
\sum_{a\in A}\frac1a
=\sum_{G\in\cG_t}\prod_{i\in G}J_i
\ge\sum_{G\in\cG_t}z^{|G|}
\ge\lambda^t.
\]
Since $t=(1-\eta)L/z+O(1)$,
\[
\lambda^t=(\log N)^{(1-\eta)\log\lambda/z+o(1)}.
\]
Letting first $N\to\infty$, then $\lambda\uparrow\wtLambda_k(z)$ and $\eta\downarrow0$, proves~\eqref{eq:weighted-lower}.
\end{proof}

\section{The exact exponent}

We now combine the two transference bounds.  Set
\[
\alpha_k:=\liminf_{N\to\infty}\frac{\log f_k(N)}{\log\log N},
\qquad
\beta_k:=\limsup_{N\to\infty}\frac{\log f_k(N)}{\log\log N}
\]
and $D_k(z):=\Lambda_k(z)-z$.  By~\eqref{eq:trivial-pressure}, $0\le D_k(z)\le1$.

\begin{theorem}[Pressure variational formulas]\label{thm:pressure-formulas}
For every $k\ge3$, the limit $\gamma_k=\alpha_k=\beta_k$ exists and
\begin{align}
\gamma_k
&=\lim_{z\to\infty}\bigl(\Lambda_k(z)-z\bigr)
=\inf_{z>0}\bigl(\Lambda_k(z)-z\bigr),\label{eq:gamma-upper-var}\\
&=\sup_{z>0}z\log\frac{\Lambda_k(z)}z
=\sup_{u>0}\frac{\log\wtLambda_k(u)}u.
\label{eq:gamma-lower-var}
\end{align}
Equivalently,
\begin{equation}\label{eq:derivative-zero}
\gamma_k=\lim_{u\downarrow0}\frac{\wtLambda_k(u)-1}{u}.
\end{equation}
\end{theorem}

\begin{proof}
Theorem~\ref{thm:upper} gives $\beta_k\le D_k(z)$ for every $z>0$.  Apply Theorem~\ref{thm:lower} with parameter $1/z$ and use~\eqref{eq:duality}; then
\[
\alpha_k\ge z\log\frac{\Lambda_k(z)}z
=z\log\left(1+\frac{D_k(z)}z\right).
\]
For $t\ge0$, $\log(1+t)\ge t-t^2/2$.  Since $0\le D_k(z)\le1$, for $z\ge1$ this yields
\[
\alpha_k\ge D_k(z)-\frac1{2z}.
\]
Consequently,
\[
\limsup_{z\to\infty}D_k(z)
\le\alpha_k\le\beta_k
\le\inf_{z>0}D_k(z)
\le\liminf_{z\to\infty}D_k(z).
\]
The reverse inequality between liminf and limsup is automatic, so all displayed quantities are equal.  This proves the existence of $\gamma_k$ and~\eqref{eq:gamma-upper-var}.

The lower bound above shows
$z\log(\Lambda_k(z)/z)\le\gamma_k$ for every $z>0$.  As $z\to\infty$,
\[
z\log\frac{\Lambda_k(z)}z
=z\log\left(1+\frac{D_k(z)}z\right)
=D_k(z)+O(z^{-1})\longrightarrow\gamma_k.
\]
This proves the first supremum formula in~\eqref{eq:gamma-lower-var}; the second follows from $u=1/z$ and~\eqref{eq:duality}.  Finally,
\[
\wtLambda_k(u)=u\Lambda_k(1/u)=1+uD_k(1/u),
\]
so~\eqref{eq:derivative-zero} follows on letting $u\downarrow0$.
\end{proof}

It remains to remove the weighted pressure from the answer.  This is where uniform tensor products give an exact finite-block formula rather than only a lower bound.

\begin{proposition}[Uniform representation of cosunflower pressure]\label{prop:uniform-pressure}
For every $z>0$,
\begin{equation}\label{eq:uniform-pressure}
\wtLambda_k(z)
=\sup_{\substack{n\ge1\\0\le r\le n}}
\bigl(M_k(n,r)z^r\bigr)^{1/n}.
\end{equation}
\end{proposition}

\begin{proof}
Every layer of a $k$-cosunflower-free family is $k$-cosunflower-free.  Thus
\[
C_k(n;z)\le\sum_{r=0}^nM_k(n,r)z^r
\le(n+1)\max_{0\le r\le n}M_k(n,r)z^r.
\]
Taking $n$th roots and passing to the limit gives the upper bound in~\eqref{eq:uniform-pressure}.

Conversely, fix $n,r$ and an extremal family $\cG\subseteq\binom{[n]}r$ of size $M_k(n,r)$.  By Lemma~\ref{lem:tensor}, its $t$-fold block product is $k$-cosunflower-free, is $tr$-uniform, and has weighted mass
$M_k(n,r)^tz^{tr}$.  Hence
\[
C_k(tn;z)\ge M_k(n,r)^tz^{tr}.
\]
Taking $(tn)$th roots and then $t\to\infty$ proves the reverse inequality.
\end{proof}

\begin{proof}[Proof of Theorem~\ref{thm:main}]
The existence of the exponent follows from Theorem~\ref{thm:pressure-formulas}.  Combining~\eqref{eq:gamma-lower-var} and~\eqref{eq:uniform-pressure}, and freely interchanging suprema, gives
\begin{align*}
\gamma_k
&=\sup_{z>0}\sup_{\substack{n\ge1\\0\le r\le n}}
\frac{\log M_k(n,r)+r\log z}{nz}.
\end{align*}
The layer $r=0$ contributes $0$.  For $r\ge1$, elementary calculus gives
\[
\sup_{z>0}\frac{\log M+r\log z}{nz}
=\frac{r}{en}M^{1/r},
\]
the maximum being attained at $z=eM^{-1/r}$.  Substituting $M=M_k(n,r)$ proves~\eqref{eq:finite-block}.
\end{proof}

\section{Consequences and finite verification}

Theorem~\ref{thm:main} immediately turns every finite uniform construction into a number-theoretic lower bound.

\begin{corollary}[Finite transfer principle]\label{cor:finite-transfer}
If $\cG\subseteq\binom{[n]}r$ is $k$-cosunflower-free and $|\cG|=M$, then
\[
f_k(N)\ge(\log N)^{\frac{r}{en}M^{1/r}-o(1)}.
\]
Conversely, for every $\varepsilon>0$ there is a finite triple $(n,r,\cG)$ of this type for which
\[
\frac{r}{en}|\cG|^{1/r}>\gamma_k-\varepsilon.
\]
\end{corollary}

Thus lower-bound searches are finite and reproducible.  For fixed $k,n,r$, let a hypergraph have vertex set $\binom{[n]}r$ and one hyperedge for each $k$-tuple forming a cosunflower.  Then $M_k(n,r)$ is its independence number, equivalently the optimum of the finite integer program
\[
\max\sum_{S\in\binom{[n]}r}x_S,
\qquad
\sum_{S\in E}x_S\le k-1\ \text{for every cosunflower }E,
\qquad x_S\in\{0,1\}.
\]
This gives certified lower bounds for $\gamma_k$ from exact integer optimization, while upper bounds require structural control of the same uniform extremal numbers.

\begin{lemma}\label{lem:complete-layer}
If $0\le r\le k-2$, then the complete layer $\binom{[n]}r$ is $k$-cosunflower-free.
\end{lemma}

\begin{proof}
Suppose $S_1,\dots,S_k$ are $r$-sets with a common pairwise union $U$.  Put $D_i=U\setminus S_i$.  The sets $D_i$ are pairwise disjoint: an element in $D_i\cap D_j$ would be absent from $S_i\cup S_j=U$.  They all have the same size $d=|U|-r$.  If the $S_i$ are distinct, then $d\ge1$, and disjointness gives
\[
kd\le|U|=r+d,
\]
so $(k-1)d\le r$.  This is impossible when $r\le k-2$.  Hence no $k$ distinct members of the layer form a cosunflower.
\end{proof}

\begin{proof}[Proof of Corollary~\ref{cor:capacity}]
Taking $r=k-2$ in Lemma~\ref{lem:complete-layer} gives
$M_k(n,r)=\binom nr$.  Theorem~\ref{thm:main} therefore yields
\[
\gamma_k\ge\lim_{n\to\infty}\frac{r}{en}\binom nr^{1/r}
=\frac{r}{e(r!)^{1/r}}.
\]
At $z=1$, equations~\eqref{eq:gamma-upper-var} and~\eqref{eq:gamma-lower-var} give
\[
\log\mu_k^{\mathrm S}
=\log\wtLambda_k(1)
\le\gamma_k
\le\Lambda_k(1)-1
=\mu_k^{\mathrm S}-1.
\]
The trivial inequality $\mu_k^{\mathrm S}\le2$ gives $\gamma_k\le1$, proving~\eqref{eq:capacity-comparison}.

For $k=3$, Deuber, Erd\H{o}s, Gunderson, Kostochka, and Meyer~\cite{DeuberEtAl1997} constructed families giving $\mu_3^{\mathrm S}>1.551$, while Naslund and Sawin~\cite{NaslundSawin2017} proved
$\mu_3^{\mathrm S}\le3/2^{2/3}$.  Hence
\[
\log(1.551)<\gamma_3\le\frac{3}{2^{2/3}}-1,
\]
which is the stated numerical interval.

If a set is LCM-$k$-free, then it is LCM-$(k+1)$-free, so $f_k(N)\le f_{k+1}(N)$ and therefore $\gamma_k\le\gamma_{k+1}$.  Finally, Stirling's formula gives
\[
(r!)^{1/r}=\frac re\exp\left(\frac{\log(2\pi r)}{2r}+O(r^{-2})\right),
\]
which implies~\eqref{eq:large-k}.  Together with $\gamma_k\le1$, this proves $\gamma_k\to1$.
\end{proof}

\begin{remark}
Complementation identifies $M_k(n,r)$ with the largest size of a $(n-r)$-uniform $k$-sunflower-free family on $[n]$.  Consequently, formula~\eqref{eq:finite-block} can be stated entirely in the usual sunflower language.  It also shows why the ordinary capacity $\mu_k^{\mathrm S}=\Lambda_k(1)$ need not by itself determine $\gamma_k$: the latter depends on the full distribution of extremal families across uniform layers.
\end{remark}

\section*{AI Declaration}

The authors acknowledge the use of AI-assisted tools (GPT -5.5 Pro ) during this research for language refinement, exploratory discussion of proof ideas, and critical feedback on draft versions. All AI-generated outputs were reviewed, verified, and revised by the authors, who bear sole responsibility for the scientific content and conclusions of this work.

\end{document}